\documentclass[11pt]{amsart}
\usepackage{amssymb}
\usepackage{amsfonts}
\usepackage{amsxtra}
\usepackage{amsmath}
\usepackage{amstext}
\usepackage{amsthm}
\usepackage{amsbsy}
\usepackage{latexsym}
\usepackage{amscd}
\usepackage{graphics}

\newtheorem{theorem}{Theorem}[section]
\newtheorem{corollary}[theorem]{Corollary}
\newtheorem{lemma}[theorem]{Lemma}
\newtheorem{proposition}[theorem]{Proposition}

\newtheorem{question}[theorem]{Question}

\theoremstyle{definition}
\newtheorem{definition}[theorem]{Definition}
\newtheorem{example}[theorem]{Example}

\theoremstyle{remark}

\newcommand{\A}{\mathcal{A}}

\newcommand{\s}{\mathcal{S}}
\newcommand{\I}{\mathcal{I}}
\newcommand{\J}{\mathcal{J}}

\newcommand{\N}{\mathcal{N}}

\newcommand{\cont}{\mathfrak{c}}

\newcommand{\M}{\mathcal{M}}
\newcommand{\Po}{\mathcal{P}}

\newcommand{\C}{\mathcal{C}}
\newcommand{\Real}{\mathbb{R}}

\newcommand{\Nat}{\mathbb{N}}

\def\proof{\removelastskip\par\medskip \noindent {\bf Proof.}\enspace}
\def\endproof{\hbox{ }\hfill{\qed}\par\medskip}

\begin{document}

\title[$\sigma$-Ideals and outer measures on the real line ]
 {$\sigma$-Ideals and outer measures on the real line }

\author{ S. Gac\'ia-Ferreira }

\address{Centro de Ciencias Matem\'aticas, Universidad Nacional Aut\'onoma de M\'exico, Campus Morelia, Apartado Postal 61-3, Santa Maria, 58089, Morelia, Michoac\'an, M\'exico}
\email{sgarcia@matmor.unam.mx}

\author{A. H. Tomita}
\address{Instituto de Matem\'atica e Estat\'istica, Universidade de S\~ao Paulo \\ Rua do Mat\~ao, 1010, CEP 05508-090, S\~ao Paulo, Brazil}
\email{tomita@ime.usp.br}

\author{ Y. F. Ortiz-Castillo }
\address{Instituto de Matem\'atica e Estat\'istica, Universidade de S\~ao Paulo \\ Rua do Mat\~ao, 1010, CEP 05508-090, S\~ao Paulo, Brazil}
\email{jazzerfoc@gmail.com}

\thanks{Research of the first-named author was supported
by  CONACYT grant no. 176202 and PAPIIT grant no. IN-101911. The second author has support from FAPESP (Brazil) Proc. 2014/16955-2.
The third author  has support from CNPq (Brazil) - ``Bolsa de Produtividade em Pesquisa, processo 307130/2013-4'', CNPq Universal 483734/2013-6}



%\date{June 25, 2011}

% -----------------------------------------------------------

\begin{abstract} A {\it weak selection} on $\Real$ is a function
$f: [\Real]^2 \to \Real$ such that $f(\{x,y\}) \in \{x,y\}$ for each $\{x,y\} \in [\Real]^2$. In this article, we continue with the study (which was initiated  in \cite{ag}) of  the outer measures $\lambda_f$ on the real line $\mathbb{R}$ defined by  weak selections $f$.  One of the main results is to show that $CH$ is equivalent to the existence of a weak selection $f$ for which:
\[
\mathcal \lambda_f(A)=
\begin{cases}
 0  & \text{ if $|A| \leq \omega$,}\\
\infty & \text{ otherwise.}
\end{cases}
\]
Some conditions are given  for a $\sigma$-ideal of $\Real$ in order to be  exactly the family $\N_f$ of $\lambda_f$-null subsets for some weak selection $f$. It is shown that there are $2^\cont$ pairwise distinct ideals on $\Real$ of the form $\N_f$, where $f$ is a weak selection. Also we prove that    Martin Axiom implies the existence of a weak selection $f$ such that $\N_f$ is exactly the $\sigma$-ideal  of meager subsets of $\Real$.
Finally, we shall study pairs of weak selections which are ``almost equal'' but they have different families of $\lambda_f$-measurable sets.
\end{abstract}

% -----------------------------------------------------------
\maketitle
% -----------------------------------------------------------

%%%%%%%%%%%%%%%%%%%%%%%%%%%%%%%%%%%%%%%%%%%%%%%%%%
\section*{Preliminaries and Introduction}
%%%%%%%%%%%%%%%%%%%%%%%%%%%%%%%%%%%%%%%%%%%%%%%%%%

For an infinite set $X$ and a cardinal number  $\kappa$, we let $[X]^\kappa = \{ F \subseteq X : |F| = \kappa \}$ and similarly we define $[X]^{\leq \kappa }$ and  $[X]^{\geq \kappa }$. The cardinality of the real line shall be denoted by $\mathfrak{c}$. The letters $\alpha$, $\beta$, $\gamma$, $\eta$, $\xi$ and $\zeta$ will represent ordinal numbers. The Greek letter $\omega$ stands for the first infinite cardinal number and $\omega_{1}$ stands for the first uncountable cardinal number. Given a fixed ordinal $\alpha$, $(\beta, \eta)$, $[\beta, \eta)$,  $(\beta, \eta]$ and $[\beta, \eta]$  will denote the interval of $\alpha $ with respect to the order of $\alpha$ for each $\beta < \eta < \alpha$. The usual order  on the real line $\Real$ will be simply denoted by $\leq$.

 \medskip

 A function $f : [X]^2 \to X$ is called a {\it weak selection} if $f(F) \in F$ for all $F \in [X]^2$. The most common example of a weak selection on the real line is the Euclidean weak selection $f_{E}:\left[\mathbb{R}\right]^{2}\to \mathbb{R}$ given by $f_{E}\left(\{x,y\}\right)=x \quad \hbox{iff} \quad x<y$, for all $\{x, y\} \in [\mathbb{R}]^2$. For a given weak selection $f$ on $X$, we say that a point $x\in X$ is {\it $f$-minimum} if $f\big(\{x, y\}\big)= x$ for every $y\in X$. For a weak selection  $f : [\Real]^2 \to \Real$ and $\{x, y\} \in [\Real]^2$, we say $x <_f y$ if $f(\{x, y\}) = x$, and for $x, y \in \mathbb{R}$ we define $x \leq_f y$ if either $x <_f y$ or $x = y$. This relation $\leq_f$ is reflexive, antisymmetric and linear, but not transitive. If $f$ is a weak selection  and $r, s \in \mathbb{R}$, then the $f$-intervals are  denoted by $(r, s)_{f}:=\Big\{x \in X \, : \, r <_f x <_{f} s\Big\}$, $(r ,s]_{f}:=\Big\{x \in X \, : \, r <_{f} x \leq_f s \Big\}$, $(r,\rightarrow)_{f}:=\Big\{x \in X \, : \, r <_f x \Big\}$ etc. For the Euclidean intervals we just write $(r ,s)$, $(r, s]$, $(r,\rightarrow)$ etc. In the notation $(r, s)$ we shall understand that $r < s$.  Meanwhile, in the general notation $(r, s)_f$ we do not necessarily  require  that $r <_f s$.

\medskip

The weak selections have been studied by several mathematicians in the areas of Topology and Analysis (see for instance \cite{c}, \cite{gt}, \cite{gmn}, \cite{gmnt}, \cite{gn01}, \cite{gn04},  \cite{hm}, \cite{hm2} and \cite{ns}).  One important property of the weak selections is that they give the possibility to  generates topologies  which have interesting topological properties (see \cite{gt}, \cite{gn01} and \cite{hm}).  In the article \cite{ag}, the authors introduced the notion of $f$-outer measure by using a weak selection  $f$  on the real line as follows:

\medskip

 If $f: [\mathbb{R}]^2 \to \mathbb{R}$ is a weak selection and $A \subseteq \mathbb{R}$, then we define
$$
\lambda_{f}(A):=\inf\Big\{\sum_{n \in \mathbb{N}} |s_{n} - r_{n}| \, : \, A \subseteq \bigcup_{n \in \mathbb{N}}(r_{n}, s_{n}]_{f} \Big\},
$$
if $A$ can be cover for a countable family of semi-open $f$-intervals, and   $\lambda_{f}(A)=+\infty$ otherwise. This function $\lambda_{f}:\mathcal{P}(\mathbb{R}) \longrightarrow [0,+\infty]$  is an outer measure on the real line $\mathbb{R}$ which  generalizes the Lebesgue outer measure. Certainly,  the Lebesgue outer measure $\lambda$ coincides with  the outer measure $\lambda_{f_E}$ (briefly denoted by
$\lambda$) where $f_E$ is the weak selection induced by the Euclidean order of $\Real$. Given a weak selection $f$ on $\Real$,  $\N_f$ will denote the $\sigma$-ideal  consisting of all $\lambda_f$-null sets   and the family of $\lambda_f$-measurable subsets will be denoted by  $\M_f$. In particular, $\M$ will stand for the family of Lebesgue measurable sets, and $\N$ for the Lebesgue null subsets of $\Real$.
Since the null sets of an outer measure form an $\sigma$-ideal, it is very natural to consider the following question.

\begin{question}\label{q0} What are the $\sigma$-ideals $I$ on $\mathbb{R}$ for which  there is a weak selection $f$ such that $I = \N_f$?
\end{question}

A particular case of this question is the following.

\begin{question}\label{q1}
Is there a weak selection $f$ such that $\N_f$ is exactly the $\sigma$-ideal of meager subsets of $\Real$?
\end{question}

In the first section, we will see that there are  $2^\cont$ many pairwise distinct ideals of the form  $\N_f$, where $f$ is a weak selection on $\Real$.
An example of a $\sigma$-ideal on $\Real$ which is not of the form $\N_f$ for any weak selection $f$ is described. Some conditions on $\sigma$-ideals are given in order be of the form $\N_f$ for some weak selection $f$. In the second section, we show that $CH$ is equivalent several conditions involving a very spacial weak selections on $\Real$.
In the third section, we use   Martin Axiom to show the existence of a weak selection $f$ for which  $\N_f$ is precisely the $\sigma$-ideal  of meager subsets of $\Real$. The last section is devoted to  study pairs of weak selections which are similar modulo a set but their induced measures  have different families of  measurable sets.

%%%%%%%%%%%%%%%%%%%%%%%%%%%%%%%%%%%%%%%%%%%%%%%%%
\section{$\N_f$-ideals}
%%%%%%%%%%%%%%%%%%%%%%%%%%%%%%%%%

First, let us construct $2^\cont$ many pairwise distinct ideals of the form  $\N_f$, where $f$ is a weak selection on $\Real$.
The construction is based on the  following theorem.

\medskip

We recall the definition of the direct sum of two ideals $\I$ and $\J$:
$$
\I\oplus \J =\{I\cup J:I\in\I \text{ and } J\in \J\}.
$$
It is easy to show that if $\I$ and $\J$ are two $\sigma$-ideals, then $\I\oplus \J$ is a $\sigma$-ideal too.

\medskip

The following easy lemma has been frequently used to construct examples and counterexamples (see for instance \cite{ag}).

\begin{lemma}\label{l1.0}
Let $A\subseteq \mathbb{R}$ and let $f$ be a weak selection on $\mathbb{R}$. Suppose that there exists a sequence $(x_n)_{n\in \Nat}$ converging to $x$,  in the Euclidean topology, such that for every $n\in \Nat$ either  $x <_f a <_f x_n$ or $x_n <_f a <_f x$, for every $a\in A$ (but at most a countable subset). Then $\lambda_f(A)=0$.
\end{lemma}

\begin{theorem}\label{l3.8} For every $X\subseteq \Real$  and for every  weak selection $f$ without minimum, there is a weak selection $g$ such that
$$
\N_f \oplus  \Po(X) = \N_g.
$$
\end{theorem}

\proof
Let $X\subseteq \Real$ and let $f$ be a weak selection. If $|\Real\setminus X|\leq \omega$,  by Corollary 2.7 from \cite{ag}, then we obtain that $\N_f \oplus  \Po(X)= \Po(\Real)$, and Example 3.6 from \cite{ag} provides a weak selection $g$ such that $\N_g = \Po(\Real)$. Hence, we will suppose that $|\Real\setminus X|> \omega$. Fix a non-trivial sequence  $\{x_n:n\in\Nat\}\subseteq \Real\setminus X$
converging to a point $x\in \Real \setminus X$. Now, define the weak selection $g$ by:
\[
g(\{r,s\})=
\begin{cases}
r &  r=x \ \text{  and} \ s\in X,  \\
s &  r=x_n \ \text{for some} \  n\in \Nat \ \text{  and} \ s\in X , \\
f(\{r,s\})  & \text{ otherwise.}
\end{cases}
\]
Assume, without loss of generality, that $g$ does not have a minimum (other\-wise, we modify the function by making simple changes). We assert that $g$ satisfies the requirements.
To show that it suffices to prove the equivalence $\lambda_g(Y)=0$ iff $Y= A\cup B$ for some $A\subseteq X$ and $B\in \N_f$. This equivalence will follow from the next claim.

\begin{flushleft}
\textbf{Claim 1:} Given  $Y\subseteq \Real\setminus X$, we have that  $\lambda_g(Y)=0$ iff $\lambda_f(Y)=0$.
\end{flushleft}

\smallskip

{\bf Proof of Claim 1:} Let $Y\subseteq \Real\setminus X$. From the definition of $g$ it is evident that $g(\{r,s\})=f(\{r,s\})$ whenever $r\in \Real\setminus X$ and $s\in \Real$. Then, $\lambda_f(Y)=0$ iff for every positive $k \in \Nat$ there are $\{ r_n^k : n \in \Nat \}$, $\{ s_n^k : n \in \Nat \} \subseteq  \Real$ such that
$$
Y\subseteq \bigcup_{n\in \Nat}(r_n^k, s_n^k]_f \text{ and } \sum_{n\in \Nat}|s_n^k - r_n^k|<\frac{1}{k},
$$
which is equivalent to say that for every positive $k \in \Nat$
 there are $\{ r_n^k : n \in \Nat \}$, $\{ s_n^k : n \in \Nat \} \subseteq  \Real$ such that
$$
Y\subseteq \bigcup_{n\in \Nat}(r_n^k, s_n^k]_g \text{ and } \sum_{n\in \Nat}|s_n^k - r_n^k|<\frac{1}{k}.
$$
Clearly, this last statement is equivalent to $\lambda_g(Y)=0$. And the claim is proved.

\medskip

Now pick $A\subseteq X$ and $B\in\N_f$. Without loss of generality we may assume that $B\subseteq \Real\setminus X$. By the definition of $g$ and Lemma \ref{l1.0}, we have that  $\lambda_g(A)=0$. Also by Claim 1, $\lambda_g(B)=0$ and then $\lambda_g(A\cup B)=0$.
On the other hand, if $\lambda_g(Y)= 0$, then $\lambda_g(Y\setminus X) =0$. Thus by Claim 1, $Y\setminus X\in \N_f$. The proof is done because of
 $Y= (Y\cap X)\cup(Y\setminus X)$.
\endproof

\begin{corollary}
There are $2^\cont$ pairwise distinct $\sigma$-ideals of the form  $\N_f$.
\end{corollary}

\proof Consider the $\sigma$-ideal of the  Lebesgue null sets  $\N$. By Theorem \ref{l3.8}, we know that for every $A\notin \N$ there is a weak selection $g$ such that $\N_g=\N\oplus \Po(A)$. It is then clear that we can find $2^\cont$-many subsets of $\mathbb{R}$ which are not in $\N$.  We know that these sets induce  distinct $\sigma$-ideals of the form $\N \oplus \Po(A)$.
\endproof

\begin{corollary}
The following statements are equivalents:
\begin{enumerate}

\item There exists a weak selection $f$ such that $\N_f = [\Real]^{\leq \omega}$.

\item For every $X\subseteq \Real$ there is a weak selection $f$ such that $\N_f = [\Real]^{\leq \omega}\oplus \Po(X)$.
\end{enumerate}
\end{corollary}

\proof The implication  $(1)\Rightarrow (2)$  follows from Lemma \ref{l3.8} and $(2)\Rightarrow (1)$ is trivial by putting  $X = \emptyset$.
\endproof

The previous corollary suggests the following question.

\begin{question}
Assume that there is $X \subseteq \Real$ such that $|\Real\setminus X| = \cont$ and a weak selection $f$ such that $\N_f = [\Real]^{\leq \omega}\oplus \Po(X)$. Under this assumption,  is it true that there exists a weak selection $g$ such that $\N_g = [\Real]^{\leq \omega}$?
\end{question}

\medskip

Now we shall give some $\sigma$-ideals on $\Real$ which are not of the form $\N_f$ for any weak selection $f$. To have this done we need to recall that the {\it cofinality} of an ideal $\mathcal{I}$, denoted by $cf(\mathcal{I})$, is the least cardinality  of a subset $\mathcal{B} \subseteq \mathcal{I}$ such that
for every $I \in \mathcal{I}$ there is $B \in \mathcal{B}$ such that $I \subseteq B$.
If $\A$ is a nonempty family of nonempty subsets of $\Real$, then
$$
\I(\A) = \{X\subseteq \Real: X\subseteq \cup \A' \text{ for some } \A'\in [\A]^{\leq\omega} \cup [\Real]^{\leq \omega} \}
$$
will denote the $\sigma$-{\it ideal generated} by $\A$. Observe that $cf(\mathcal{I}(\A)) \leq |\A|$ for every   nonempty family $\A$ of nonempty subsets of $\Real$

\medskip

The following result is well-known, but we would like to include a proof of it.

\begin{lemma}\label{toml1.1}
For every weak selection $f$, $cf(\N_f) \leq \cont$.
\end{lemma}

\proof Let $\A$ be the family of all subsets of $\Real$ of the form:
$$
\bigcap_{k\in \Nat} \big(\bigcup_{n\in \Nat}(r_n^k, s_n^k]_f \big), \text{ where  } \{(r_n^k, s_n^k]_f:n, k\in \Nat\}  \text{ satisfies }
$$
$$
\sum_{n\in \Nat}|s_n^k - r_n^k|< \frac{1}{k} \text{ for each } k\in \Nat.
$$
Since there are  only $\cont$-many pairwise different $f$-intervals, we must have that  $|\A|\leq \cont$. By the definition of $\lambda_f$ it is evident that $\N_f= \I(\A)$.
\endproof

 Given an infinite cardinal number $\kappa$, an infinite  family $\A\subseteq [X]^{\kappa}$ on an infinite set $X$ is called {\it $\kappa$-almost disjoint}, for short $\kappa-AD$-family, if $|A\cap B|<\kappa$ for  distinct  $A$, $B\in \A$. Notice that if $\A$ is a $\kappa-AD$-family on $\Real$ and $cf(\kappa) > \omega$,  then $cf(\mathcal{I}(\A)) = |\A|$.

\begin{example}  Let $\kappa$ be the least cardinal  such that $\cont^\kappa > \cont$. Then we have that $\omega_1 \leq \kappa \leq \cont$ and  $\cont^{< \kappa} = \cont$. Identified $\Real$ with $\cont^{< \kappa}$. Then, for each $s \in \cont^\kappa$ define $A_s = \{ s |_\alpha : \alpha < \kappa \} \subseteq \cont^{<\kappa}$. Consider the $\sigma$-ideal $\mathcal{I}(\A)$ on $\cont^{< \kappa}$ generated by the family $\A = \{ A_s : s \in \cont^\kappa \}$.  Since $cf(\mathcal{I}(\A))\geq |\A| = \cont^\kappa > \cont$, we must have that $\mathcal{I}(\A) \neq \N_f$ for every weak selection $f$.

It is clear that if   $\A$ is a maximal $\cont-AD$-family on $\Real$, then $cf(\cont) < |\A| \leq 2^{\cont}$. Hence, if $\cont$ is regular and  $\A$ is a maximal $\cont-AD$-family on $\Real$, then $\mathcal{I}(\A) \neq \N_f$ for every weak selection $f$

Consider the $\sigma$-ideal $\mathcal{J}$ of non-stationary subsets of $\cont$.
If $\cont$ is regular, then we know that  $cf(\mathcal{J}) > \cont$ and hence   $\mathcal{J} \neq \N_f$ for every weak selection $f$.
\end{example}

The next question seems to be very natural.

\begin{question}\label{ad}  Given  an arbitrary family $\A$ of infinite subsets of  $\Real$ of size $\cont$, is  there a model of $ZFC$ in which  $\I(\A)  = \N_f$ for some weak selection $f$?
\end{question}

The answer to Question  \ref{ad} is affirmative under $CH$ as it will be shown in Theorem \ref{tomt1.2}.

\medskip

Next, we shall study the ideals of the form $[\Real]^{< \cont}\oplus \I(\A)$ where $\emptyset \neq \A \subseteq \mathcal{P}(\Real) \setminus \{\emptyset\}$.
 The following theorem will provide some conditions which guarantee that the $\sigma$-ideal $[\Real]^{< \cont}\oplus \I(\A)$ is of the form $\mathcal{N}_f$. First, we prove some preliminary results.

\begin{lemma}\label{toml1.2}
For every weak selection $f$ and for every infinite set $D\subseteq\Real $ there is a weak selection $g$ such that
$D$ does not have $g$-minimum  and
$$
|\big( (r,s]_f\setminus  (r,s]_{g}\big) \cup \big( (r,s]_{g}\setminus (r,s]_f\big)|\leq \omega
$$
for each $r$, $s\in \Real$.
\end{lemma}

\begin{proof} Suppose that $d$ is the $f$-minimum inside of $D$. Take a countable infinite subset $N \subseteq  D$ which contains $d$. Enumerate $N$ as $\{r_n: n \in\Nat\}$ such that $d= r_0
$. The   weak selection $g$ is defined as follows:

\smallskip

$g( \{r_n, r_m\} )= r_n$ whenever that $n > m$ and   $g$ is equal to $f$ on the other pairs of points.

\smallskip

\noindent It is evident that $D$ does not have a $g$-minimum and since $g$ is equal to $f$ except for a countable subset, the second property holds.
\end{proof}

Let $\alpha $ be an ordinal number of cardinality $\cont$ and let $\phi:[0, \alpha)\to \Real$ be a bijection. Then, we  transfer the order of $\alpha$ to $\Real$ by using $\phi$ and we define the weak selection $f_\phi$ on $\Real$ by the rule $f_\phi(\{r , s\})= r$ iff $\phi^{-1}(r) <_\alpha \phi^{-1}(s)$. When the characteristics  of the bijection $\phi$ is not relevant we will simply  denote by $f_\alpha$ the corresponding weak selection.

\begin{lemma}\label{l0.tom} Let $f_\cont$ be a weak selection induced by a bijection $\phi: [0,\cont) \to \Real$.
Then the following statements are equivalent for every set $A\subseteq \Real$.
\begin{enumerate}
\item $\lambda_{f_{\cont}}(A)=0$,

\item $\lambda_{f_{\cont}}(A)< \infty$, and

\item there is $r\in \Real$ such that $a \leq_{f_\cont}r $ for all $a\in A$.
\end{enumerate}
\end{lemma}

\proof  Without loss of generality we may assume that $A\subseteq \Real\setminus \{\phi(0)\}$ since $\lambda_{f_{\cont}}(\{\phi(0)\})=0$. The implication $(1) \Rightarrow (2)$ is evident.

$(2) \Rightarrow (3) $.  Since $\lambda_{f_{\cont}}(A)< \infty$ there is a family $\big\{\{a_n, b_n\}: n, \ m\in \Nat\big\} \subseteq [\Real]^2$ such that
$A  \subseteq \bigcup_{m\in \Nat}(a_n, b_n]_{f_\cont} $. Let $\xi= sup \{\phi^{-1}(b_n): n \in \Nat\}$. As $\phi^{-1}[A]\subseteq [0,\xi]$, we obtain that  $a \leq_{f_\cont} \phi(\xi)$ for all $a\in A$.

$(3) \Rightarrow (1)$. Since $|A|<\cont $ we can find a sequence $(x_n)_{n\in \Nat}$ convergent to $\phi(0)$ such that $\phi^{-1}(x_n) > \phi^{-1}(r)$ for each $n\in \Nat$. Then $A\subseteq (\phi(0), x_n]_{f_\cont}$ for each $n\in \Nat$. By Lemma \ref{l1.0}, $\lambda_{f_{\cont}}(A)=0$.
\endproof

We remark that Lemma \ref{l0.tom} could fail for ordinals greater than $\cont$. For an example we may consider the ordinal $\cont+1$ and any bijection $\phi:[0, \cont]\to \Real$. Following similar arguments to the implication $(3) \Rightarrow (1)$, we may choose a sequence $(x_n)_{n\in \Nat}$ convergent to $\phi(\cont)$. Set $A=\{r\in\Real: \forall n \in \Nat (x_n <_{f_\phi} r) \}$ and observe that $|A|=\cont$ and
$\lambda_{f_{\phi}}(A)=0$. Hence, we deduce that if there is $A \in \mathcal{N}_{f_\alpha} \cap [\Real]^\cont$, then $\cont < \alpha$.

\begin{theorem}\label{tomt1.1}{\bf [$\cont$ is regular]}
Let $\A = \{A_\xi : \xi<\cont\}$ be a family of nonempty subsets of  $\Real$ such that:

$(i)$ $|\Real \setminus \big(\bigcup_{\eta < \xi}A_\eta \big)|=\cont$ for every $\xi < \cont$, and

$(ii)$ $\Real =\bigcup_{\xi < \cont}A_\xi$.

\noindent   Then there is a weak selection $f$ such that $\N_f=[\Real]^{< \cont}\oplus \I( \A)$.
\end{theorem}

\proof Assume that $\cont$ is regular. Enumerate $[\cont]^\omega$ as $\{ N_\xi : \xi < \cont \}$ and set $B_\xi = \bigcup_{\eta \in N_\xi}A_\eta$ for each $\xi < \cont$. By using the following fact:

{\it every uncountable subset of $\Real$  contains a nontrivial convergent sequence and its limit point,}

\noindent for each $\xi < \cont$ we can find  $x_\xi \in \Real$ and a nontrivial sequence $S_\xi$ of  $\Real$ so that:
\begin{enumerate}
\item $S_\xi$  converges to $x_\xi$,

\item $x_\xi \notin S_\xi$, and

\item $\big[B_\xi \cup \big(\bigcup_{\eta < \xi}(B_\eta \cup S_\eta \cup \{x_\eta\}) \big) \big]\cap \big( S_\xi \cup\{x_\xi\}\big)= \emptyset$.
\end{enumerate}
Let us define the weak selection $f$ as follows:
\[
f(\{r,s\})=
\begin{cases}
s &  r\in B_\xi \text{  and} \ s \in S_\xi, \text{ for some }  \xi <  \cont, \\
r &  r\in B_\xi   \text{  and} \ s = x_\xi, \text{ for some } \xi<\cont, \\
f_\cont(\{r,s\})  & \text{ otherwise.}
\end{cases}
\]
First, we prove that $\N_f\subseteq[\Real]^{< \cont}\oplus \I(\A)$.
Let $M\in \N_f$. Then there is a countable family of $f$-intervals $\{(r_n^k, s_n^k]_f: n\in \Nat, k\in\Nat\}$ such that
$$
M\subseteq \bigcup_{n\in \Nat}(r_n^k, s_n^k]_f \text{ and } \sum_{n\in \Nat}|s_n^k-r_n^k|< \frac{1}{k} \text{ for each } k\in \Nat.
$$
As $\Real= \bigcup_{\eta < \cont}A_\eta$, there is $\xi < \cont$ such that
$$
\{r_n^k: n, k\in\Nat\} \cup \{s_n^k: n, k\in\Nat\} \subseteq B_\xi.
$$
Let $M' = M \setminus \big(B_\xi \cup \{x_\xi\} \cup S_\xi\big)$. By the definition of $f$, we know that $f(\{x,y\})=f_\cont(\{x,y\})$ for every   $x\in B_\xi \cup \{x_\xi\} \cup S_\xi$ and every$y\in M'$.  Hence,  $M'\subseteq \bigcup_{n, k\in\Nat} (r_n^k, s_n^k]_{f_\cont} $ and so  $\phi^{-1}(M') \subseteq [0,s]$, where $s = sup\{  s_n^k : n, k \in \Nat \}$, which implies that $|M'|<\cont$. This shows that  $\N_f \subseteq[ \Real]^{< \cont}\oplus \I( \A)$. On the other hand, it is evident that $\I( \A) \subseteq \N_f$. Fix $L \in[\Real]^{< \cont}$. By applying Lemma \ref{l0.tom} and the regularity of $\cont$, we have that $\lambda_{f_\cont}(L) = 0$.  As above, choose a countable family of real numbers $\{ r_n^k, s_n^k : n, k\in\Nat\}$ such that
$$
L \subseteq \bigcup_{n\in \Nat}(r_n^k, s_n^k]_{f_\cont} \text{ and } \sum_{n\in \Nat}|s_n^k-r_n^k|< \frac{1}{k} \text{ for each } k\in \Nat.
$$
Let $\zeta < \cont$ be such that $\{ r_n^k, s_n^k : n, k\in\Nat\} \subseteq B_\zeta$. Consider the set $L' = L \setminus \big(B_\zeta \cup \{x_\zeta\} \cup S_\zeta \big)$. Observe that
$f(\{x,y\})=f_\cont(\{x,y\})$ for every   $x\in B_\zeta \cup \{x_\zeta \} \cup S_\zeta$ and every $y\in L'$. This implies that  $L'\subseteq \bigcup_{n, k\in\Nat} (r_n^k, s_n^k]_{f} $ and so
$L' \in \N_f$ and hence $L\in \N_f$. Thus, $[\Real]^{< \cont}\oplus \I( \A) \subseteq N_f$.
\endproof

The following question is somehow related to Theorem \ref{l3.8}.

\begin{question} Given two weak selection $f$ and $g$ without minimum, is there a weak selection $h$ such that
$$
\N_f \oplus \N_g = \N_h?
$$
\end{question}

In the next section, we will see  that the previous question has a positive answer under $CH$.

\medskip

We remark that there are $\sigma$-ideals on $\Real$ which are not of the form $\mathcal{I}_{\A}$ for a $\cont$-$AD$-family $\A$. For instance the $\sigma$-ideal of meager subsets of the real line. This ideal will be consider in the next sections.

%%%%%%%%%%%%%%%%%%%%%%%%%%%%%%%%%%%
%%%%%%%%%%%%%%%%%%%%%%%%%%%%%%%%%%%
\section{$\N_f$-ideals under $CH$}
%%%%%%%%%%%%%%%%%%%%%%%%%%%%%%%%%

Let us consider in this section the trivial measure on $\Real$ defined by:
\[
\mathcal \mu(A)=
\begin{cases}
 0  & \text{ if $|A| \leq \omega$,}\\
\infty & \text{ otherwise.}
\end{cases}
\]
It is evident from Lemma \ref{l0.tom} that  $CH$ implies that $\lambda_{f_\cont}=\mu$ and so $\N_{f_\cont} =  [\Real]^{\leq \omega}$. In the next theorem,
we will show that $CH$ is equivalent to the existence of a weak selection $f$ for which  $\lambda_{f_\cont}=\mu$.
This theorem will be a consequence of the following lemmas.

\begin{definition}
Let $\alpha \leq \cont$ be an ordinal number. We say that a weak selection $f$ {\it generates an $\alpha$-ordered set } if there is an indexed subset $\{r_\beta: \beta<\alpha\}\subseteq \Real$ such that either  $r_\beta<_f r_\gamma$ whenever $\beta < \gamma < \alpha$, or $r_\gamma <_f r_\beta$ whenever $\beta < \gamma < \alpha$.
\end{definition}

Given $X\subseteq {\mathbb R}$, $r \in \Real$ and a weak selection $f$, we set
$$
L(r)^X_f=(\leftarrow , r )_{f} \cap X \ {\rm and} \ R(r)^X_f=(r, \rightarrow )_{f} \cap X.
$$
In particular, we have that $L(r)^{\Real}_f = (\leftarrow , r )_{f} $ and  $R(r)^{\Real}_f= (r, \rightarrow )_{f}$ for every $r \in \Real$.

\begin{lemma}\label{l02.tom}
Suppose that $f$ is a weak selection on $\Real$ for which there exist $X\in [\Real]^{\geq \omega_2}$ and $Y\in [X]^{\geq \omega_2}$ such that,  for each $r\in Y$, we have that  either $|L(r)^{X}_f|\leq \omega_1$ or $|R(r)^{X}_f|\leq \omega_1$. Then $f$ generates an $\omega_2$-orderer set in $Y$.
 \end{lemma}

 \proof Assume that $f$, $X$ and $Y$  satisfy the hypothesis. Without loss of generality we may assume that $|L(r)^{X}_f|\leq \omega_1$ for each $r\in Y$, the case when $|R(r)^{X}_f|\leq \omega_1$ for each $r\in Y$, is managed analogously. We will find recursively the points of the $\omega_2$-ordered set. Pick $r_0 \in Y$ arbitrarily. Of course the set $\{r_0\}$ is an $1$-ordered set in $Y$. Let $\alpha < \omega_2$ and assume that, for each $\beta <\alpha$, a real number  $r_\beta \in Y$ has been chosen   such that $r_\gamma < r_{\beta} $ whenever $\gamma < \beta < \alpha$. By hypothesis, we know that the set  $X_\alpha = \bigcup_{\beta <\alpha}L(r_\beta)^X_{f}$ has cardinality at most $\omega_1$. As $|Y| \geq \omega_2$, we may find $r_\alpha \in Y \setminus X_\alpha$.  It is then clear that $r_\beta < r_\alpha$ for each $\beta < \alpha$ and hence the set $\{r_\beta: \beta \leq \alpha \}$ is an $(\alpha + 1)$-ordered set. By continuing this construction, we obtain a set  $B=\{r_\alpha: \alpha < \omega_2\}$ which is an $\omega_2$-ordered set in $Y$.
 \endproof

 \begin{lemma}\label{l01.tom}
Assume that $\cont\geq \omega_2$. If $f$ is a weak selection which generates an $\omega_2$-ordered set $A$ in $\Real$, then $A$ contains  subset $X$ of size $\omega_1$ such that  $\lambda_f(X)=0$.
 \end{lemma}

 \proof Let $A=\{r_\xi : \xi<\omega_2 \}$ be the $\omega_2$-ordered set generated by $f$. Without loss of generality, assume that $A$ does not have isolated points, in the Euclidean topology, and that $r_\xi<_f r_\zeta$ provided that  $\xi < \zeta < \omega_2$, the other case is analogous. Since $A$ is separable in the Euclidean topology, we can choose a subset $I\in [\omega_2]^{\leq \omega}$ such that $\{r_\xi : \xi \in I\}$ is dense in $A$.  Set $\beta = \min\{I\}$ and $\gamma = \sup\{I\} $.  Then we have that  $I \subseteq [\beta,\gamma]$ and hence $|[\beta , \gamma]|=\omega_1$.  Our required set is  $X=\{ r_\xi : \xi \in (\beta , \gamma]\}$.  Indeed,  since $A$ is an $\omega_2$-ordered set, $X\subseteq (r_\beta , r_\gamma]_f$.  Fix  $\epsilon > 0$.  As the set $\{r_\xi : \xi \in I\}$ is dense in $X$,  we can choose  $\xi_1$ as the least $\xi \in [\beta,\gamma)$ such that $|r_{\xi} - r_\gamma| < \frac{\epsilon}{2}$. By using this process finitely many times we can find a finite set $\{\xi_i : i \leq l\} \subseteq I \cup \{\beta, \gamma\}$ such that $\gamma =\xi_0$,  $\beta = \xi_l$, $X \subseteq \bigcup_{i< l}(r_{\xi_i},r_{\xi_{i+1}}]_f$ and $|r_{\xi_{i+1}} - r_{\xi_i}| < \frac{\epsilon}{2^{i+1}}$, for each $i < l$. Then, we obtain that  $\lambda_f(X)\leq \epsilon$. Therefore, the $\lambda_f$-outer measure of $X$ is equal to  $0$.
 \endproof

 \begin{lemma} \label{Lemma1.tom}
For every $X\subseteq \Real$ and for every weak selection $f$, either:

$(i)$ $X$ contains a  subset of size $\omega_1$ and zero $\lambda_f$-outer measure, or

$(ii)$ $\{ a \in X:\, |L(a)^X_f|\leq \omega_1 \text{ or } |R(a)^X_f|\leq \omega_1\}$ has cardinality at most $\omega_1$.
\end{lemma}

\begin{proof} It is evident that, if $|X| \leq \omega_1$, then $(ii)$ holds. Assume that $|X| > \omega_1$ and suppose, without loss of generality, that  $A=\{ a \in X:\, |L(a)^X_f|\leq \omega_1 \}$ has cardinality  at least $\omega_2$. By Lemma \ref{l02.tom}, $f$ generates an $\omega_2$-ordered set contained in $A$. Hence, by Lemma \ref{l01.tom}, $A$ contains a subset of size $\omega_1$ and $\lambda_f$-outer measure $0$.
\end{proof}

The next lemma  follows directly from Theorem 2.5 and Corollary 2.6 of  \cite{ag}.

\begin{lemma}\label{l00.tom}
For every weak selection $f$ and for every $r\in \Real$, we have that:
\[
\mathcal \lambda_f(\{r\})=
\begin{cases}
0 & \text{   $r$ is not $f$-minimum} \\
\infty  & \text{ otherwise.}
\end{cases}
\]
\end{lemma}

We  are ready to state the main result of this section.

\begin{theorem}\label{t1.tom}
The following statements are equivalent:
\begin{enumerate}
\item $CH$

\item There is a weak selection $f$ such that $| (\leftarrow , r )_{f}|\leq \omega$ for each $r\in \Real$.

\item There is a weak selection $f$ such that $|(r, \rightarrow )_{f}|\leq \omega$ for each $r\in \Real$.

\item There is a weak selection $f$ such that every  $f$-interval $(r,s]_f$ is countable for each $r$, $s\in \Real$.

\item There is a weak selection $f$ and a subset $D\subseteq \Real$ such that $|D|=\mathfrak{c}$ and  $| (r,s]_f \cap D|\leq\omega$ for every  $r$, $s\in \Real$.

\item There is a weak selection $f$ such that, $\lambda_f = \mu$.
\end{enumerate}
\end{theorem}

\proof For the implication $(1) \Rightarrow (2)$ we use the weak selection $f_\cont = f_{\omega_1}$. For the implications
$(2) \Rightarrow (3)$ we consider the opposite weak selection $\hat{f}$ of $f$ which is defined by $\hat{f}(\{x, y\}) = y$ iff $f(\{x, y\}) = x$. The implications $(3) \Rightarrow (4)$ and $(4) \Rightarrow (5)$ are straightforward.

$(5) \Rightarrow (6)$. Let $f$ be a weak selection and let $D$ be a subset of $\Real$ which satisfy the conditions required by (5). By Lemma \ref{toml1.2}, without loss of generality, we may assume that $D$ does not have $f$-minimum.
Fix a bijection $\phi:D\to \Real$.  Define the weak selection $g: [\Real]^2\to \Real$ by the rule: $\phi(r)<_{g} \phi (s)$ iff $r <_f s$. Then the following conditions hold:

$(i)$ all the $g$-intervals are countable and

$(ii)$ there is not a $g$-minimum.

\noindent Clause $(i)$ implies that  $\lambda_g(A) = \infty$ for every uncountable $A \subseteq \Real$ and, by $(ii)$ and Lemma \ref{l00.tom}, $\lambda_g(A) = 0$ for every $A\in [\Real]^{\leq \omega}$.

$(6) \Rightarrow (4)$. Let $f$ a weak selection such that $\lambda_f = \mu$. Suppose that there are two real numbers $r$ and $s$ such that $|(r, s]_f|>\omega$. Then $(r,s]_f \notin \N_{\lambda_f}$, but  $\mu \big((r,s]_f\big)  =\lambda_f\big((r,s]_f\big)   \leq |s-r|<\infty$ which is impossible.

$(4) \Rightarrow (1)$. Let $f$ be a weak selection such that every  $f$-interval of the form $(r,s]_f$ is countable. Fix $A\in \N_{\lambda_f}$. Then there is a countable family of ordered pairs
$$
\{(r_n, s_n): r_n, \ s_n \in \Real, \ {\rm for }  \ {\rm each } \ n\in \Nat \}
$$
such that $A\subseteq \bigcup_{n\in \Nat} (r_n, s_n]_f$. Thus $|A|\leq \omega$. This shows that $\Real$ does not have an uncountable subset of zero $f$-outer measure. Now, assume the negation of $CH$.  Then, by Lemma \ref{Lemma1.tom}, we must have that
$$
 |\{ r \in \Real:\, |(\leftarrow , r )_{f}|\leq \omega_1 \text{ or } |(r, \rightarrow  )_{f}|\leq \omega_1\}|\leq \omega_1.
$$
\noindent Then, there is a real number $a$ such that $|(a , \rightarrow  )_{f}|\geq \omega_2$. By applying again the same lemma to the set $(a , \rightarrow  )_{f}$ we obtain
$$
|\{ r \in (a , \rightarrow  )_{f}:\, |L(r)^{(a , \rightarrow  )_{f}}_f|\leq \omega_1 \text{ or } |R(r)^{(a , \rightarrow  )_{f}}_f|\leq \omega_1\}|\leq \omega_1.
$$
Thus there is $b \in (a , \rightarrow  )_{f}$ such that  $|L(b)^{(a , \rightarrow  )_{f}}_f|\geq \omega_2$. As $L(b)^{(a , \rightarrow  )_{f}}_f\subseteq (a,b]_f$, $|(a,b]_f|\geq \omega_2$ which is a contradiction. Therefore, the $CH$ holds.
\endproof

Base on the results already established ,  we shall say that a weak selection $f$ satisfies the {\it countable null condition ($c.n.c.$)} if $\N_f=[\Real]^{\leq \omega}$.

 \begin{corollary}\label{ct1.tom}
[$CH$] There is a weak selection $f$ which satisfies the $c.n.c.$.
\end{corollary}

In connection with the last corollary, we shall consider the  following statement.

\smallskip

$CNH$ ({\it Countable Null Hypothesis}): There is a weak selection $f$ on $\Real$ which satisfies de $c.n.c.$

\medskip

We are unable to answer the following question.

\begin{question}
Is  $CNH$ equivalent to  $CH$?
\end{question}

Under $CH$ the $\sigma$-ideals with cofinallity $\leq \cont$ may be characterized as follows.

\begin{theorem}\label{tomt1.2} {\bf [CH]} If $\mathcal{I}$ is a $\sigma$-ideal on $\Real$ with $cf(\mathcal{I}) \leq \cont$, then
there is a weak selection $f$ such that $\N_f = \mathcal{I} $.
\end{theorem}

\proof. Suppose $CH$ and  let $\A = \{A_\xi: \xi<\omega_1\}$ be a family of nonempty subsets of $\Real$ that generates the $\sigma$-ideal $\mathcal{I}$.  Assume, without loss of generality, that $|\Real \setminus \bigcup_{\eta < \xi}A_\xi|=\omega_1$  for every $\xi<\omega_1$. In the opposite, we have that $\I( \A) = \mathcal{I} =\Po(\Real)$,  for this case the weak selection defined in Example 3.6 from \cite{ag} does the job.  Then, we have that   the family $\A$ satisfies conditions $(i)$-$(ii)$ from Theorem \ref{tomt1.1}, then, there exists a weak selection $f$ such that $\N_f=[\Real]^{< \cont}\oplus \I( \A) = [\Real]^{\leq\omega}\oplus \I( \A) =  \mathcal{I}$.
\endproof

\begin{corollary}\label{tomc1t1.2}{\bf [CH]}
There exists a weak selection $f$ such that $\N_f$ is exactly the $\sigma$-ideal of of meager subsets of $\Real$.
\end{corollary}

\proof. The statement follows from the fact that the $\sigma$-ideal of meager subsets of $\Real$ is generated by the family of all nowhere dense closed subsets of $\Real$ which has  size $\cont$.
\endproof

\begin{corollary}\label{tomc2t1}{\bf [CH]}
For every pair of weak selections $f$ and $g$ there is a weak selection $h$ such that $\N_h=\N_f\oplus \N_g$.
\end{corollary}

%%%%%%%%%%%%%%%%%%%%%
\section{$\N_f$-ideals under $MA$}
%%%%%%%%%%%%%%%Marti A

It is well known that, under Martin Axiom, either the union of $< \cont$ meager sets is meager, and the union of $< \cont$ Lebesgue measure zero sets has measure $0$.
However, this is not the case for  our outer measures that we have considered so far. Indeed, let $\A$ a partition of $\Real$ such that $|\A|=\omega_1$ and $|A|=\cont$ for every $A\in\A$. By using some ideas from the proof of  Theorem \ref{tomt1.1} and the weak selection $f_\cont$, we can find  a weak selection $f$ such that  $\N_f=\Nat^{\leq \omega}\oplus \I(\A) = \I(\A)$ and $\Real\notin \N_f$. Hence, in a model of $MA + \neg CH$ this $\sigma$-ideal $\N_f$ is not closed under $< \cont$ unions.

\medskip

Recall that an ideal $\mathcal{I}$ is a $<\cont$-{\it ideal} if it is closed under unions of subfamilies of size $<\cont$.

\begin{theorem}\label{tomt2.1}
Let $\I$ be a proper $<\cont$-ideal of $\Real$ containing $\Real^{<\cont}$ with   $cf(\I) = \cont$. Then there is a weak selection $f$ such that $\N_f= \I$.
\end{theorem}

\proof  Suppose that  $\A = \{A_\xi: \xi<\cont\}$ generates the ideal $\I$. By transfinite induction, for each $\xi < \cont$, we will define points $x_\xi$ and sets $S_\xi$ and $X_\xi$  such that:
\begin{enumerate}
\item $S_\xi$ is a nontrivial sequence converging to $x_\xi$ for each $\xi < \cont$,

\item $X_0=A_0$, $X_\xi= \big(\bigcup_{\eta < \xi} X_\eta\big)\cup A_{\xi}$ if $0<\xi<\cont$ is limit and $X_{\xi+1}= X_\xi\cup A_{\xi+1}\cup S_\xi \cup\{x_\xi\}$ for each $\xi < \cont$,  and

\item $S_\xi \cup x_\xi\subseteq \Real \setminus \Big( A_\xi\cup \big(\bigcup_{\eta < \xi} X_\eta\big)  \Big)$ for each $\xi < \cont$.

\item $|X_\xi\setminus \bigcup_{\eta\leq\xi+1}A_\eta|<\cont$ for each $\xi < \cont$.
\end{enumerate}

Then, the first stage is determined by $(1)$-$(3)$. Let $\xi< \cont$ and suppose that we have defined the points $x_\eta$ and the sets $S_\eta$ and $X_\eta$ for every $\eta < \xi$ satisfying  $(1)$-$(3)$. Since $\I$ is a proper $<\cont$-ideal containing $\Real^{<\cont}$, we have that  $|\Real \setminus \Big(A_\xi \cup \big(\bigcup_{\eta < \xi} X_\eta\big)\Big)|=\cont$. Then we may chosse $x_\xi$ and $S_\xi$ which satisfy conditions $(1)$ and $(3)$. Finally, the set $X_\xi$ is defined by $(2)$ and  $(4)$  follows from the construction.

\medskip

Now let define the weak selection $f$ as follows:
\[
f(\{r,s\})=
\begin{cases}
r &  r\in X_\xi \ \text{  and} \ s=x_\xi, \text{ for some } \xi<\cont, \\
s &  r\in X_\xi \ \text{  and} \ s\in S_\xi, \text{ for some } \xi<\cont, \\
f_\cont(\{r,s\})  & \text{ otherwise.}
\end{cases}
\]
It is not hard to show that $f$ does not have a minimum. Observe that the weak selection $f$ is well-defined because of $(3)$. Furthermore, $\lambda_f\big(X_\xi\big)=0$ for every $\xi<\cont$. So $\I\subseteq \N_f$. To finish the proof we will show that $\N_f\subseteq\I$. In fact, fix $B\in \N_f$ and let $\{(r_n^k, s_n^k]_f: n\in \Nat, k\in\Nat\}$ be a countable family of $f$-intervals such that
$$
B\subseteq \bigcup_{n\in \Nat}(r_n^k, s_n^k]_f \text{ and } \sum_{n\in \Nat}|s_n^k-r_n^k|< \frac{1}{k} \text{ for each } k\in \Nat.
$$
Choose  $\xi <\cont$ so that $\{r_n^k: n, k\in\Nat\} \cup \{s_n^k: n, k\in\Nat\} \subseteq X_\xi$. By clause $(4)$, we have that  $|X_\xi\setminus \bigcup_{\eta\leq\xi+1}A_\eta|<\cont$ and hence $X_\xi\in \I$. Thus, $B\cap X_\xi\in\I$. Let $B'=B\setminus X_\xi$.
\medskip

\begin{flushleft}
\textbf{Claim 1:} $B'\cap (S_\eta\cup\{x_\eta\}) = \emptyset$ for each ordinal $\eta$ satisfying $\xi \leq \eta < \cont$.
\end{flushleft}

\smallskip

{\bf Proof of Claim 1:} Fix $b\in B'$. Since $r_n^k$, $s_n^k\in X_\xi$, we have that  $s <_f r_n^k<_f x_\eta$ and  $s <_f s_n^k<_f x_\eta$,  for every $k, n \in\Nat$ and  for every $s\in S_\eta$, where  $\xi \leq \eta < \cont$. So $b\notin \bigcup_{\xi\leq\eta}(S_\eta\cup\{x_\eta\})$ because of for every $k\in\Nat$ there is $n\in\Nat$ such that $r^k_n < b \leq  s^k_n$.  This shows that  $B'\cap (S_\eta\cup\{x_\eta\}) = \emptyset$ for each $\xi \leq \eta < \cont$.

\medskip

By the definition of $f$ and Claim 1, we obtain that $f(\{b, x\})=f_\cont(\{b, x\})$ for every $b\in B'$ and every $x\in X_\xi$. Then $\lambda_f(B')=\lambda_{f_\cont}(B')=0$ which implies that $B'\in[\Real]^{<\cont}$. So $B'\in\I$. Therefore, $N_f\subseteq\I$.
\endproof

\begin{corollary}\label{tomc1t2.1} {\bf [MA]}
If  $\I$ is a $<\cont$-ideal of $\Real$ with $cf(\I) = \cont$,  then there is a weak selection $f$ such that $\N_f= \I$.
\end{corollary}

\begin{corollary}\label{tomc2t2.1} {\bf [MA]}
There exists a weak selection $f$ such that $\N_f$ is exactly the $\sigma$-ideal  of meager subsets of $\Real$.
\end{corollary}

\begin{question} In $ZFC$, is there a weak selection $f$ such that $\N_f$ is exactly the $\sigma$-ideal  of meager subsets of $\Real$?
\end{question}

%%%%%%%%%%%%%%%%%%%%%%%%%%%%%%%%%%%%%%%%%%%%%%%%%%%%%%%%%%%%%%%%%%%%%%%%%%%%%%%%
\section{Equivalence and congruence of weak selections under their outer measures}
%%%%%%%%%%%%%%%%%%%%%%%%%%%%%%%%%%%%%%%%%%%%%%%%%%%%%%%%%%%%%%%%%%%%%%%%%%%

In this section, we shall study  some conditions on two weak selections $f$ and $g$ in order to induce the same family of measurable sets of $\Real$. Indeed, we will call this property {\it  $\M$-equivalence}. The approximation to the $\M$-equivalence property that we propose in the present paper is inspired in Theorem 1.12 from \cite{ag}, which establishes that if  $f$ and $g$ are two weak selections for which there exists $N \in[\Real ]^{\omega}$ so that $g\big(\{r,s\}\big)=r$ iff $r <_f s$ when $|\{r,s\}\cap N| \leq 1$, then $\M_{f} = \M_{g}$. This property of the weak selections $f$ and $g$  somehow  can be  formalized in the next definition.

\begin{definition}\label{d0}
Let $f$ and $g$ be two weak selections and let $N\subseteq \Real$. We say that $f$ and  $g$ are:
\begin{enumerate}
\item  {\it Congruent  mod $N$},  in symbols $f\cong_{N}g$, if $r <_g s$ iff $r <_f s$ whenever $|\{r,s\}\cap N|\leq 1$.

\item {\it Weakly congruent  mod $N$},  in symbols $f\cong_{N}^*g$, if $r <_g s$ iff $r <_f s$ whenever  $\{r,s\}\cap N=\emptyset$.
\end{enumerate}
\end{definition}
It is then natural to ask whether or not  two weak selections $f$ and $g$ are $\M$-equivalent provide $f$ and $g$ are congruent (weakly congruent) mod $N$ for certain $N\subseteq \Real$. In what follows, we will be only interested on weak selections which are weakly congruent mod a countable set. First of all, we shall describe two weak selections $f$ and $g$ such that $f\cong_{N}^*g$ for a finite set $N$ such that $\M_{f} \neq \M_{g}$. To show this we will prove that, for a given  fix point $x\in \Real$, every weak selection $f$ is weakly congruent mod $\{x\}$ with another  weak selection $g$ for which the real line has zero $g$-outer measure. A usual way to define this kind of outer measures uses Lemma \ref{l1.0}, but unfortunately this lemma is not strong enough to prove the promised result. Thus we shall need to prove a stronger lemma.

\begin{lemma}\label{l1}
Let $A\subseteq \Real$ and let $f$ be a weak selection. Suppose that there exists a sequence $(x_n)_{n\in \Nat}$ converging to $x$  in the Euclidean topology such that
$$
\sum_{n = 0}^\infty |x - x_n| < \infty
$$
and
$$
A\subseteq \Big(\bigcup_{n = k}^\infty (x, x_n]_f \Big) \cup \Big(\bigcup_{n = k}^\infty (x_n, x]_f \Big),
$$
for every $k \in \Nat$. Then $A$ has zero $f$-outer measure.
\end{lemma}

\proof  Let  $(x_n)_{n\in \Nat}$ be a sequence converging to $x$ satisfying the required conditions. Fix $\epsilon > 0$ arbitrary and let $k\in \Nat$ such that
$$
\sum_{n = k}^\infty |x - x_n| < \frac{\epsilon}{2}.
$$
By hypothesis,
$$
A\subseteq \Big(\bigcup_{n = k}^\infty (x, x_n]_f \Big) \cup \Big(\bigcup_{n = k}^\infty (x_n, x]_f \Big).
$$
Then
$$
\lambda_f(A)\leq 2 \sum_{n = k}^\infty |x - x_n| < \epsilon.
$$
Since $\epsilon$ was chosen arbitrarily we obtain that $\lambda_f(A)=0$.
\endproof

\begin{theorem}\label{t1}
For every weak selection $f$ and for every $x\in \Real$ there exists a weak selection $g$ such that $\lambda_g(\Real)=0$ and  $g\cong_{\{x\}}^* f$.
\end{theorem}

\proof Fix a weak selection $f$ and choose any point $x$ and a sequence $(x_n)_{n\in\Nat}$  converging to $x$ satisfying
$\sum_{n = 0}^\infty |x - x_n| < \frac{\epsilon}{2}$ and $0 < |x - x_n|$ for every $n \in \Nat$. Enumerate  the family of all nontrivial subsequences of  $(x_n)_{n\in\Nat}$ by $\{S_\xi:\xi<\cont\}$. For each $\xi<\cont$, define
$$
L_\xi=\{r\in \Real : r<_f s\ {\rm for}  \ {\rm each}  \  s\in S_\xi \},   \ R_\xi=\{r\in \Real : s<_f r\ {\rm for}  \ {\rm each}  \  s\in S_\xi \}
$$
$$
{\rm and} \ \ X_\xi=L_\xi\cup R_\xi.
$$
The required weak selection $g$ must satisfy that $g|_{[\Real\setminus \{x\}]^2}=f|_{[\Real\setminus \{x\}]^2}$.  Let us see that $\Real = \bigcup_{\xi < \cont}X_\xi$. Indeed,
pick $r\in \Real$ arbitrary and consider the sets   $L = \{n\in \Nat: r <_f x_n\}$ and $R = \{n\in \Nat: x_n  <_f r\}$. Then we have that  at least one of these two sets is infinite. If  $|L|=\omega$  and $\xi<\cont$ is such that $S_\xi= (x_n)_{n\in L}$,  then it is evident that $r \in L_\xi$. A similar conclusion is obtained for the case when  $R$ is infinite.
We shall define  $g$ over $\{\{x,r\}: r\in \Real\setminus\{x\}\}$ since on the other points $g$ will agree with the weak selection $f$. Fix $r \in \Real \setminus \{x\}$ and let $\xi < \cont$ the least ordinal for which $x \in X_\xi$. If $r \in L_\xi$, then we define
$x <_g r$ , and if $r \in R_\xi$, then we define $r <_g x$. It is clear that $g$ is a  well-defined weak selection such that $g\cong_{\{x\}}^* f$ . According to  Lemma \ref{l1}, we obtain that
$\lambda_f(\Real)=0$.
\endproof

\begin{corollary}\label{m} For every weak selection $f$  with $\mathcal{P}(X) \neq \M_f$ and for every $x\in \Real$ there is a weak selection $g$ such that   $g\cong_{\{x\}}^* f$ and $\M_g = \mathcal{P}(\Real) \neq \M_f$.
\end{corollary}

\begin{lemma}\label{l3.3}
Let $f$ be a weak selection. Suppose that $X\in [\Real]^{\cont}$ satisfies that
$$
|X\setminus \bigcup_{n\in\Nat} (r_{n}, s_n]_{f} |=\cont,
$$
for every countable family of $f$-intervals  $\{(r_{n}, s_n]_{f} : n \in \mathbb{N} \}$.
Then $X$ can be partitioned  in two disjoint  sets each one of them  can not be covered  for any countable family of $f$-intervals.
\end{lemma}

\proof
We shall use transfinite induction to define the required sets. Enumerate $([\mathbb{R}]^2)^{\omega} $ as $\big\{\s_{\xi} \, : \, \xi<\mathfrak{c}\big\}$ and  each $\s_{\xi}$ as $\{
(r^{\xi}_{n}, s^\xi_n) : n \in \mathbb{N} \}$. Choose $x_{0}^1$ and $x_{0}^2$ two different points  in $X \setminus \bigcup_{n\in\Nat} (r^0_{n}, s^0_n]_{f}$ and suppose that for each
$\zeta < \xi$ we have carefully selected two real number
$$
x_\zeta^1, \ x_\zeta^2 \in X\setminus \big(\bigcup_{n\in\Nat} (r^\zeta_{n}, s^\zeta_n]_{f} \cup \{x_\eta^{i}: \eta < \zeta \ {\rm and} \ i= 1, 2\}\big).
$$
By hypothesis, we have that
$$
| X\setminus \big( \bigcup_{n\in\Nat} (r^\xi_{n}, s^\xi_n]_{f} \cup \{x_\zeta^{i}: \zeta < \xi \ {\rm and} \ i= 1, 2\} \big)|=\cont.
$$
So we may choose two distinct real numbers $x_{\xi}^{1} $ and $x_{\xi}^{2} $ in this set. This ends with the induction. Now let
$$
X_1=\{x_\xi^1:\xi < \cont\} \text{ and } X_2=X\setminus X_1.
$$
It is evident that $\{ X_1, X_2\}$ is a partition of $X$  and $\{x_\xi^2:\xi < \cont\}\subseteq X_2$.  Let $\{(r_{n}, s_n]_{f} : n \in \mathbb{N} \}$ be a countable family of $f$-intervals and let $\xi<\cont$ be such that $r_n=r_n^\xi$ and $s_n=s_n^\xi$ for every $n\in \Nat$. It is then clear that  the family $\{(r_{n}, s_n]_{f} : n \in \mathbb{N} \}$ cannot  cover  none of the sets $X_1$ nor $X_2$  since $\{x_\xi^1, x_\xi^2\}\cap \big(\bigcup_{n \in \Nat}(r_{n}, s_n]_{f} \big)=\emptyset$.
\endproof

\begin{theorem}
Let $f$ be a weak selection. Assume that  $X\in \M_f$ satisfies the hypothesis of Lemma \ref{l3.3}. Then for every pair of distinct points $a$, $b\in \Real$ there is a weak selection $g$ such that $f\cong_{\{a, b\}}^* g$, $\lambda(X)= |b - a|$ and $X\notin \M_g$.
\end{theorem}

\proof
Suppose that $X\subseteq \Real $ satisfies that $|X\setminus \bigcup_{n\in\Nat} (r_{n}, s_n]_{f} |=\cont$ for every countable family of $f$-intervals  $\{(r_{n}, s_n]_{f} : n \in \mathbb{N} \}$. Let $a$, $b\in \Real$ be  distinct with $a < b$. Define $g$ as $a <_g x <_g b$ for every $x\in X$ and $g(\{y, z\}) := f(\{y, z\})$  otherwise.  It is evident that $g\cong_{\{a, b\}}^* f$. By Lemma \ref{l3.3}, there is a partition $\{X_1, X_2\}$ of $X$ so that each one of the sets  can not be covered  for any countable family of $f$-intervals.  Since the interval $(a,b]_g$ contains the set $X$, we must have that $\lambda_g(X) \leq b-a$ and hence  $\lambda_g(X_1) \leq b-a$ and $\lambda_g(X_2)\leq b-a$. Let $\{(r_{n}, s_n]_{g} : n \in \mathbb{N} \}$ be a countable family of $g$-intervals.  If $(a,b]_g\in \{(r_{n}, s_n]_{g} : n \in \mathbb{N} \}$, then
$$
\sum_{n\in \Nat} |s_n - r_n| \geq b- a.
$$
Hence, assume that $(a,b]_g\notin \{(r_{n}, s_n]_{g} : n \in \mathbb{N} \}$. Following the notation from Lemma \ref{l3.3}, there is  $\xi<\cont$ such that $r_n=r_n^\xi$ and $s_n=s_n^\xi$ for every $n\in \Nat$.  Besides, we have that  either $r_{n}^{\xi}\neq a$ or $s_{n}^{\xi}\neq b$ for every $n\in \Nat$.
On the other hand  since $x_\xi^i \notin \bigcup_{n \in \Nat}(r_n^\xi,s_n^\xi]_f$, we have that  that either $x_\xi^i <_f r_{n}^{\xi}$ or $s_{n}^{\xi} <_f x_\xi^i$, for every $i= 1, 2$ and $n\in \Nat$. Thus, we obtain that  either $x_\xi^i <_g r_{n}^{\xi}$ or $s_{n}^{\xi} <_g x_\xi^i$ for every $i= 1, 2$ and  $n\in \Nat$. That is, $x_\xi^i \notin \bigcup_{n \in \Nat}(r_n^\xi,s_n^\xi]_f$ for every $i= 1, 2$ and $n\in \Nat$. All these results imply that  $\lambda_g(X)=\lambda_g(X_1)=\lambda_g(X_2)= b-a$. As $X= X_1\cup X_2$, then  we obtain that none of the these three sets can be $g$-measurable.
\endproof

Now we turn out  our attention to the weak selection $f_\cont$, where $\leq_{f_\cont}$ is order isomorphic to the order of $\cont$. Observe that, since $\M_{f_\cont}=\Po(\Real)$, we can not apply Corollary \ref{ct1} to find a weak selection $g$ such that $g\cong_{N}^*f_\cont$, for some $N\in[\Real]^{\omega}$, and $\M_g\neq\M_{f_\cont}$. We will see in the next theorem that, in general, Corollary \ref{ct1} could fail if we do not require the condition $\Po(\Real)\setminus \M_f\neq \emptyset$ by showing  that every weak selection $g$, weakly congruent mod a countable set with $f_\cont$ satisfy that  $\M_{g}=\Po(\Real)$.  To prove it we need two lemmas.

\begin{lemma}\label{l2.1}
Let $f$ a weak selection such that $\lambda_f\big((r, s ]_f\big)=0$ for every $r, s \in \Real$. Then $\M_f= \Po(\Real)$ and $\lambda_f(X)\in \{0, + \infty \}$ for each $X\subseteq \Real$.
\end{lemma}
\proof Let $X\subseteq \Real$. Of course if $X$ can not be cover by a countable cover of $f$-intervals, then $\lambda_f(X)=+ \infty$ and so $X\in\M_f$. Assume that there exists a countable family $\{(r_n, s_n]_f: n\in \Nat\}$ such that $X\subseteq \bigcup_{n\in \Nat} (r_n, s_n]_f$. Then
$$
\lambda_f(X)\leq \sum_{n\in\Nat}\lambda_f\big((r_n, s_n ]_f\big)=0.
$$
Of course, in this case, we also obtain that $X\in \M_f$.
\endproof

\begin{theorem}\label{tomt.3.1}
For every ordinal number $\alpha$ of cardinality $\cont$,  and for every bijection $\phi:[0, \alpha)\to \Real$, we have that $\M_{f_\phi}=\Po(\Real)$.
\end{theorem}

\proof Assume that  $\cont \leq \alpha < \cont^+$ and fix a bijection $\phi:[0, \alpha)\to \Real$. We will prove   by transfinite induction that $(\phi(0), \phi(\xi)]_{f_\phi}\in \N_{f_\phi}$  for every ordinal $\xi$ with $0< \xi < \alpha$. First observe that $|(\phi(0), \phi(\xi)]_{f_\phi}|<\cont$ for all $\xi < \cont$. Hence,  we can find  a sequence $S$ which converges to $\phi(0)$ such that $S\cap (\phi(0), \phi(\xi)]_{f_\phi}=\emptyset$ and $\phi(\xi) <_{f_\phi} s$ for all $s \in S$. So, by Lemma \ref{l1.0}, $\lambda_{f_\phi}\big((\phi(0), \phi(\xi)]_{f_\phi}\big)=0$.  Thus we may assume that $\cont \leq \xi < \alpha$ and   $(\phi(0), \phi(\eta)]_{f_\phi}\in \N_{f_\phi}$ for every $\eta < \xi$. If $\xi= \eta+1$, then the assertion follows directly from the inductive hypothesis because of $(\phi(0), \phi(\xi)]_{f_\phi}=(\phi(0), \phi(\eta)]_{f_\phi}\cup\{\phi(\xi)\}$. Now, assume that $\xi$ is a limit ordinal. We need to consider two cases:

\smallskip

{\bf Case I:} $cf(\xi)=\omega$. Let $(\eta_n)_{n\in\Nat}$ be a strictly sequence of ordinals such that  $\eta_n\nearrow\xi$. Then,
$$
(\phi(0), \phi(\xi)]_{f_\phi}=\big(\bigcup_{n\in \Nat}(\phi(0), \phi(\eta_n)]_{f_\phi}\big)\cup \{\phi(\xi)\}.
$$
So, the inductive hypothesis implies that $\lambda_{f_\phi}\big((\phi(0), \phi(\xi)]_{f_\phi}\big)=0$.

\smallskip

{\bf Case II:} $cf(\xi)>\omega$. Suppose that $[\phi(\xi), \phi(\alpha))_{f_\phi}$ is not a closed set in the Euclidean topology of $\Real$. Then there are an ordinal number $\eta<\xi$ and a sequence  of ordinals $(\eta_n)_{n\in\Nat}$, with  $ \xi\leq \eta_n $ for each $n\in\Nat$, so that the sequence  $(\phi(\eta_n))_{n\in\Nat}$ converges to $\phi(\eta)$ in the Euclidean topology. By Lemma \ref{l1.0},  we have that $\lambda_{f_\phi}\big((\phi(\eta), \phi(\xi)]_{f_\phi}\big)=0$. Since the inductive hypothesis guarantees  that   $\lambda_{f_\phi}\big((\phi(0), \phi(\eta)]_{f_\phi}\big)=0$, we  obtain that $\lambda_{f_\phi}\big((\phi(0), \phi(\xi)]_{f_\phi}\big)=0$. Hence, we can suppose that  $[\phi(\xi), \phi(\alpha))_{f_\phi}$ is a closed set in the Euclidean topology of $\Real$. So,
$(\phi(0), \phi(\xi))_{f_\phi}$ is an open set in the Euclidean topology. Let $\eta\geq\xi$ such that $\phi(\eta)$ is an accumulation point, in the Euclidean topology, of $(\phi(0), \phi(\xi))_{f_\phi}$.
Then, choose a sequence of ordinals  $(\eta_n)_{n\in\Nat}$ in the interval $(0,\xi)$ so that $(\phi(\eta_n))_{n\in\Nat}$ converges to $\phi(\eta)$ in $\Real$.
Let $\eta= sup \{\eta_n:n\in\Nat\}$. Since $cf(\xi)>\omega$, we have that $\eta < \xi$. Then, by  Lemma \ref{l1.0}, we have that $\lambda_{f_\phi}\big([\phi(\eta), \phi(\xi)]_{f_\phi}\big)=0$.
By assumption we know that $\lambda_{f_\phi}\big((\phi(0), \phi(\eta)]_{f_\phi}\big)=0$. Thus,   we finally obtain that $\lambda_{f_\phi}\big((\phi(0), \phi(\xi)]_{f_\phi}\big)=0$. According to  Lemma \ref{l2.1}, $\M_{f_\phi}=\Po(\Real)$.
\endproof

Theorem \ref{tomt.3.1} provides an infinite family of distinct weak selections which are pairwise $\M$-equivalent. In particular, if $f$ is the weak selection defined in \cite[Ex. 3.6]{ag}, then $\lambda_f\big((r, s]_f\big)=\lambda_{f_\cont}\big((r, s]_f\big)=0$, for every $r, s \in \Real$, but $\lambda_f\neq \lambda_{f_\cont}$ (this answers negatively  Question 3.14 of \cite{ag}).

\medskip

 We are convinced that the particular problem of find weak selections which are $\M$-equivalent to $f_E$ seems to be very interesting. Somehow related to this, we have the following proposition.

\begin{proposition}\label{l4.1}
Let $A\subseteq \Real$ and let $f$ and $g$ be two weak selection such that $f\cong_A g$. Then
\begin{enumerate}
\item $(r,s]_f\setminus A= (r,s]_g\setminus A$, for every $r$, $s\in \Real$,  and  $(r,s]_f= (r,s]_g$, for every $r$, $s\in \Real\setminus A$; and

\item  $\lambda_f(X\setminus A)=\lambda_g(X\setminus A)$ for every $X\subseteq\Real$ such that $\lambda_f(X\setminus A)<\infty$.
\end{enumerate}
\end{proposition}

\proof By definition, statement $(1)$  follows directly from the fact that  $ r <_f x <_f s$ iff $ r <_g x <_g s$, for every $r$, $s\in \Real$ and for each $x\notin A$.
To prove $(2)$  let $X\subseteq \Real$ and suppose  that $\lambda_f(X\setminus A)<\infty$. Consider an arbitrary  $\C=\big\{\{a_n, b_n\}: n, \ m\in \Nat\big\} \subseteq [\Real]^2$.
Observe  from clause $(1)$ that
$$
X\setminus A  \subseteq \bigcup_{m\in \Nat}(a_n, b_n]_f \Leftrightarrow X\setminus A  \subseteq  \bigcup_{m\in \Nat}\big((a_n, b_n]_f\setminus A\big)
$$
$$
\Leftrightarrow X\setminus A  \subseteq  \bigcup_{m\in \Nat}\big((a_n, b_n]_g\setminus A\big) \Leftrightarrow X\setminus A  \subseteq \bigcup_{m\in \Nat}(a_n, b_n]_g.
$$
By hypothesis there exists a countable family of $f$-intervals of $X$ which witnesses that $\lambda_f(X\setminus A)<\infty$. Following the observation we have that $\lambda_g(X\setminus A)\leq\lambda_f(X\setminus A)<\infty$.  Again, by the same observation, $\lambda_f(X\setminus A)\leq\lambda_g(X\setminus A)$. Therefore, $\lambda_f(X\setminus A)=\lambda_g(X\setminus A)$.
\endproof

\begin{question}\label{q3.1}
It is true that two weak selections $f$ and $g$ are $\M$-equivalent provide  $f\cong_N g$ for some $N\in \N_f\cap \N_g$?
\end{question}

%%%%%%%%%%%%%%%%%%%%%%%%%%%%%%%%%%%%%%%%%%%%%

\end{document}